\title{Autonomous second-order ODEs: a geometric approach}
\author{A. J. Pan-Collantes \thanks{
    Department of Mathematics, Universidad de C\'{a}diz, Puerto Real, Spain;\\ 
    email: \texttt{antonio.pan@uca.es}}
\and
J. A. Alvarez-Garcia \thanks{
          Department of Mathematics, IES Jorge Juan, San Fernando, Spain;\\ 
          email: \texttt{jose.alvg@gmail.com}
                  }
        }
\documentclass{article}
\usepackage[a4paper, total={5.25in, 8in}]{geometry}
\usepackage{graphicx}
\usepackage{tikz-cd}
\usepackage[utf8]{inputenc}  
\usepackage{hyperref}
\hypersetup{
	colorlinks=true,
	linkcolor=magenta,
    filecolor=magenta,      
    urlcolor=cyan,
    citecolor=magenta,
    }

\usepackage{amssymb, amsmath}
\usepackage{amsthm}

\theoremstyle{remark}

\theoremstyle{definition}
\newtheorem{theorem}{Theorem}[section]

\newtheorem{proposition}{Proposition}[section]
\newtheorem{corollary}{Corollary}[section]
\newtheorem{definition}{Definition}[section]
\newtheorem{remark}{Remark}[section]
\newtheorem{example}{Example}[section]

\newenvironment{keywords}{}{}

\newcommand{\contract}{\,\lrcorner\,}

\usepackage{todonotes}
\usepackage{xcolor}

\usepackage{marginnote}
\usepackage{tikz}

\begin{document}

\newpage

\maketitle

\begin{abstract}
Given an autonomous second-order ordinary differential equation (ODE), we define a Riemannian metric on an open subset of the first-order jet bundle. A relationship is established between the solutions of the ODE and the geodesic curves with respect to the defined metric. We introduce the notion of energy foliation for autonomous ODEs, and highlight its connection to the classical energy concept. Additionally, we explore the geometry of the leaves of the foliation. Finally, the results are applied to the analysis of Lagrangian mechanical systems. In particular, we provide an autonomous Lagrangian for the damped harmonic oscillator.
\end{abstract}





\section{Introduction} 

Second-order ordinary differential equations (ODEs) are an essential tool for modeling a wide range of nonlinear evolutionary phenomena, especially in systems with one-dimensional dynamics. These equations arise in various fields, including physics, engineering, and biology, where they are employed to describe processes as diverse as mechanical vibrations, chemical reactions, and population dynamics.

However, solving second-order ODEs is often challenging, as no universal algorithm exists to determine their solutions. Over the past few decades, significant research has been dedicated to finding solutions and first integrals for such ODEs \cite{cheb1999integrating,duarte2001solving,muriel2009,yumaguzhin2010differential,pancinf-sym,duartefirstintegrals}. Additionally, extensive research has focused on the qualitative analysis and numerical methods for solving second-order ODEs \cite{al2014research,langkah2017numerical,al2020three}.

In most cases, researchers have focused on specific families of second-order ODEs to gain meaningful insights. In this paper, we focus on autonomous second-order ODEs, whose study remains an active area of research (see \cite{santanaAutonomous} and references therein). These equations are particularly interesting because many of them arise from dynamical systems governed by (not necessarily autonomous) Lagrangians. In particular, understanding the behavior of these equations can lead to deeper insights into the underlying mechanics of physical systems, making them an important topic in both theoretical and applied mathematics.

Recently, there has been a growing interest in associating Riemannian metrics with differential equations \cite{Bayrakdar2018a,Bayrakdar2018,bayrakdar2019geometric,bayrakdar2021curvature,bayrakdar2022geometry,pancollantes2023surfaces,pancollantes2024integration,Bayrakdar2024}, as this approach offers insights into the behavior of these equations. Following this approach, and after introducing some preliminaries in Section \ref{SecPrelim}, we show in Section \ref{SecMetric} that autonomous second-order ODEs induce a Riemannian metric on an open subset of the first-order jet bundle $J^1(\mathbb{R}, \mathbb{R})$. We then explore the geometry of the resulting Riemannian manifold, relating it to the integrability of the ODE. Specifically, we link solutions of the ODE to geodesics of the manifold. We also introduce a minimal foliation in the manifold (Section \ref{SecEnergy}), whose leaves correspond to constant energy surfaces in the context of mechanical systems, and study the geometry of these leaves. In Section \ref{SecLagrangian}, we shift our focus to ODEs derived from Lagrangian systems, illustrating, in particular, how the damped harmonic oscillator can be framed within an autonomous Lagrangian formulation.

\section{Preliminaries}\label{SecPrelim}

\subsection{Jet bundles and second-order ODEs}
A central tool in the study of ODEs is the use of jet bundles, which provide a natural geometric context for understanding the structure and solutions of these equations. In this work we will use the first-order jet bundle $J^1(\mathbb R,\mathbb R)$ to study autonomous second-order ODEs that can be written in the form
\begin{equation}\label{ODE1vez}
    u_2=\phi(u,u_1).
\end{equation}
Here, $(x,u,u_1)$ stand for the standard coordinates of $J^1(\mathbb R,\mathbb R)$, with $x$ and $u$ representing the independent and dependent variables, respectively, and $u_1$ denoting the first derivative of $u$ with respect to $x$. Also, $\phi$ denotes a smooth function defined on an open subset $U\subseteq J^1(\mathbb R,\mathbb R)$. The contact form on $J^1(\mathbb{R}, \mathbb{R})$ is defined as
\begin{equation}\label{contactform}
    \theta = - u_1 dx+du,
\end{equation}
and it captures the first-order differential relations between $u$, $x$, and $u_1$ \cite{saunders1989geometry,olver86}.

Recall that the vector field associated to equation \eqref{ODE1vez},
\begin{equation}\label{campoA}
    \partial_x+u_1\partial_u+\phi\partial_{u_1},
\end{equation}
which is defined on $U$, encodes all the relevant information about the equation, in the following sense. Given a smooth function $f:I\subseteq \mathbb R \to \mathbb R$, its first-order prolongation \cite{saunders1989geometry} is the curve $\mbox{j}^1 f: I \to J^1(\mathbb R,\mathbb R)$ defined by the expression
$$
(\mbox{j}^1 f)(x)=(x,f(x),f'(x)).
$$
It turns out that a smooth function $f$ is a solution to equation \eqref{ODE1vez} if and only if its first-order prolongation $\mbox{j}^1 f$ is an integral curve of the vector field \eqref{campoA} (see \cite{olver86,stephani} for the details).

\subsection{Riemannian geometry}
Recall that a Riemannian manifold is a smooth manifold $M$ equipped with a Riemannian metric $g$, i.e., a two-times covariant symmetric tensor field, positive definite, and hence non-degenerate. Every Riemannian metric gives rise to a uniquely determined torsionless metric connection $\nabla$, called the Levi-Civita connection.

Consider a 3-dimensional Riemannian manifold $(M,g)$, with an orthonormal frame $(e_1,e_2,e_3)$, and its corresponding dual coframe $(\omega^1,\omega^2,\omega^3)$. The connection form of the Levi-Civita connection is defined as the matrix of 1-forms $\Theta=(\Theta^i_{\,\,j})$, satisfying
\begin{equation}\label{covderivative}
    \nabla_{e_i}e_j=\sum_{k=1}^3 e_i\contract \Theta^k_{\,\,j} e_k, \quad 1\leq i,j\leq 3,
\end{equation} 
where $\contract$ denotes the interior product.

These 1-forms can be obtained from Cartan's first structural equation
\begin{equation}\label{firstcartan}
    d\omega^i = \sum_{k=1}^3 \omega^k \wedge \Theta^i_{\,\, k},\quad 1\leq i\leq 3,
\end{equation}
and the condition $\Theta^i_{\,\, j} = - \Theta^j_{\,\, i}$ (derived from the orthonormality of the frame). Details can be found, for instance, in \cite{chen1999lectures,Morita,ivey2016cartan}. 

On the other hand, recall that the notion of geodesic is used in differential geometry to extend to arbitrary spaces the idea of a straight line in flat spaces. A curve $\gamma: I\subseteq \mathbb R \to M$ is called a geodesic if
$$
\nabla_{\dot{\gamma}(t)}\dot{\gamma}(t)=0
$$
for every $t\in I$. It turns out that if a vector field $X$ satisfies $\nabla_X X=0$ then its integral curves are geodesics of the manifold \cite{lee2006curvature,docarmoriemannian,Carinena2023}.

Moreover, the behavior of geodesics is influenced by the curvature of the manifold. In Riemannian geometry, the curvature plays a crucial role in determining how geodesics diverge or converge, giving insight into the local and global geometry of the space. The curvature of the manifold is encoded in the Levi-Civita connection and can be described through the curvature 2-forms $\Omega^i_{\,\, j}$, as defined by Cartan's second structure equation
\begin{equation}\label{curvformdef}
    \Omega^i_{\, j} = d \Theta^i_{\, j} +\sum_{k=1}^3 \Theta^i_{\, k} \wedge \Theta^k_{\, j}.
\end{equation}
The components of the Riemann curvature tensor in the given frame are related to the curvauture 2-forms by the expression \cite{Morita,lee2006curvature}
\begin{equation}\label{cartansecond}
    R^i_{\,jab}=\Omega^i_{\, j} (e_a,e_b).
\end{equation}

The sectional curvature along specific planes within the tangent space can be computed from the components of the Riemann curvature tensor. For instance, the sectional curvatures along the planes spanned by the pairs $\{e_1,e_2\}$, $\{e_1,e_3\}$, and $\{e_2,e_3\}$ are given by the components $R^1_{\,\,212}$, $R^1_{\,\,313}$ and $R^2_{\,\,323}$, respectively. These components provide essential geometric information about the curvature of the manifold along the chosen planes.

Now, consider a 2-dimensional manifold $\Sigma$ embedded into $M$, in such a way that the given frame $(e_1,e_2,e_3)$ is adapted to $\Sigma$, i.e., $\omega^3|_{T\Sigma}=0$. The surface $\Sigma$ inherits a Riemannian metric from the ambient manifold, with its corresponding Levi-Civita connection. We will denote by $\tilde{\omega}^1,\tilde{\omega}^2$ the restrictions of $\omega^1,\omega^2$ to $T\Sigma$, and by $\tilde{\Theta}$ and $\tilde{\Omega}$ the connection forms and the curvature forms, respectively, of the inherited connection. 

By restriction to $T\Sigma$ of the first Cartan equation \eqref{firstcartan} we have
$$
\tilde{\Theta}^i_{\,\,j}=\Theta^i_{\,\,j} |_{T\Sigma}, \quad i,j=1,2.
$$
In addition, there must exist smooth functions $s_{ij}$ defined on $\Sigma$, such that 
\begin{equation}
    \begin{aligned}
    \Theta^3_{\,\,1}|_{T\Sigma} &= s_{11} \tilde{\omega}^1 + s_{12} \tilde{\omega}^2, \\
    \Theta^3_{\,\,2}|_{T\Sigma} &= s_{21} \tilde{\omega}^1 + s_{22} \tilde{\omega}^2,
    \end{aligned}
\end{equation}    
with $s_{12}=s_{21}$. Observe that we can express the functions $s_{ij}$ in terms of the connection forms as follows:
\begin{equation}\label{seconfform}
    s_{ij}=\tilde{e}_j \contract \Theta^3_{\,\,i}|_{T\Sigma}, \quad i,j=1,2,
\end{equation}
where $\tilde{e}_j=e_j|_{T\Sigma}$.
The reader may refer to \cite{spivak1999comprehensive3} for further details.

The shape operator $S$ is defined as 
$$
S=\sum_{i,j=1}^2 s_{ij} \tilde{\omega}^i\otimes \tilde{e}_j,
$$ 
and it has two independent invariants: the extrinsic Gaussian curvature,
\begin{equation}
    K_{\text{ext}}=s_{11}s_{22}-s_{12}s_{21},
\end{equation}
and the mean curvature,
\begin{equation}
    H=\frac{1}{2}(s_{11}+s_{22}).
\end{equation}

Regarding the intrinsic geometry of $\Sigma$, denote by $\tilde{R}$ the Riemann curvature tensor of $\Sigma$ with respect to the inherited metric. The curvature 2-form $\tilde{\Omega}$ satisfies, by definition,   
$$
\begin{aligned}
    \tilde{\Omega}^i_{\,j}&=\Omega^i_{\,j}|_{T\Sigma}+\Theta_1^3|_{T\Sigma}\wedge \Theta_2^3|_{T\Sigma}\\
    &=\Omega^i_{\,j}|_{T\Sigma}+K_{\text{ext}}\tilde{\omega}^1\wedge \tilde{\omega}^2.
\end{aligned}
$$
From here, and according to equation \eqref{cartansecond}, it is obtained Gauss' equation,
\begin{equation}\label{gausseq}
    K_{\text{\text{int}}}=R^1_{\,\,212}+K_{\text{ext}},
\end{equation}
where $K_{\text{int}}=\tilde{R}^1_{\,\,212}$ is the intrinsic Gaussian curvature of $\Sigma$.

\section{Riemannian metric associated to autonomous second-order ODEs}\label{SecMetric}

In this section we introduce the notion of Riemannian manifold associated with a given autonomous second-order ODE.
\begin{definition}\label{def:asurf}
Consider a second-order ODE in the form \eqref{ODE1vez}. We define the associated 3-dimensional Riemannian manifold as the open submanifold given by 
$$
M=\{(x,u,u_1)\in U: u_1\neq 0\}\subseteq J^1(\mathbb R,\mathbb R),
$$ 
endowed with the Riemannian metric
\begin{equation}\label{metric}
g=(1 + u_1^2) dx^2 - 2 u_1 dx du + \left( 1 + \frac{\phi^2}{u_1^2} \right) du^2 - 2 \frac{\phi}{u_1} du du_1 + du_1^2.
\end{equation}
\end{definition}

Interestingly, the vector field \eqref{campoA} associated with equation \eqref{ODE1vez} has constant unit length, as can easily be verified. Moreover, the vector field $\partial_{u_1}$ also has unit length, and it is orthogonal to the associated vector field. By standard procedures, we can complete this pair of vector fields, to get the orthonormal frame $( e_1, e_2,e_3)$ defined by
\begin{equation}\label{frame}
    \begin{aligned}
        & e_1=\partial_x+u_1 \partial_u+\phi \partial_{u_1},\\
        & e_2=\partial_u+\frac{\phi}{u_1} \partial_{u_1},\\
        & e_3=\partial_{u_1}.
        \end{aligned}
\end{equation}

The corresponding dual coframe $(\omega^1,\omega^2,\omega^3)$ is given by the 1-forms
\begin{equation}\label{coframe}
    \begin{aligned}
& \omega^1=dx,\\
& \omega^2=-u_1 dx+du,\\
& \omega^3=-\frac{\phi}{u_1} du+du_1,
\end{aligned}
\end{equation}
being $\omega^2$ the contact form $\theta$ of $J^1(\mathbb R,\mathbb R)$ given in equation \eqref{contactform}. Observe that
\begin{equation}\label{d_omegas}
    \begin{aligned}
     d \omega^1&=0,\\
     d \omega^2&=\frac{\phi}{u_1}  \omega^1 \wedge \omega^2+ \omega^1 \wedge \omega^3,\\
     d \omega^3&=\left(\phi_{u_1}-\frac{\phi}{u_1}\right) \omega^1 \wedge \omega^3+\frac{u_1 \phi_{u_1}-\phi}{u_1^2} \omega^2 \wedge \omega^3,
    \end{aligned}
\end{equation}
where subscripts denote partial derivatives (this notation will be used throughout the paper). From \eqref{d_omegas}, and using Cartan's first structural equation \eqref{firstcartan}, we obtain the connection form
\begin{equation}\label{connectionform}  
    \Theta=
\begin{pmatrix}
0 &-\frac{\phi}{u_1} \omega^2-\frac{1}{2} \omega^3 & -\frac{1}{2} \omega^2-\left(\phi_{u_1}-\frac{\phi}{u_1}\right) \omega^3\\ 
\frac{\phi}{u_1} \omega^2+\frac{1}{2} \omega^3 &0 &-\frac{1}{2} \omega^1-\frac{u_1\phi_{u_1}-\phi}{u_1^2}\omega^3\\
\frac{1}{2} \omega^2+\left(\phi_{u_1}-\frac{\phi}{u_1}\right) \omega^3 & \frac{1}{2} \omega^1+\frac{u_1\phi_{u_1}-\phi}{u_1^2}\omega^3 &0
\end{pmatrix}.
\end{equation}

We are now in a position to discuss the link between the geodesic curves of the Riemannian manifold $(M,g)$ and the solutions to equation \eqref{ODE1vez}. Recall that the integral curves of $e_1$ are precisely the first-order prolongation of solutions of equation \eqref{ODE1vez}. And, on the other hand, a straightforward calculation yields
\begin{equation}\label{geodVF}
    \nabla_{e_1}e_1=e_1\contract \Theta^k_{\,\,1} e_k=0,
\end{equation}
so the integral curves of $e_1$ are geodesic curves. Therefore, given a smooth function $f$ which is a solution to equation \eqref{ODE1vez}, then the curve $\mbox{j}^1 f$ is a geodesic of the manifold. The converse is not true, but we can establish the following weaker version:

\begin{proposition}\label{geodesicas}
    Suppose equation \eqref{ODE1vez} satisfies $\big(\frac{\phi}{u_1}\big)_{u_1}\neq 0$. If a smooth function $f$ is such that the curve $\mbox{j}^1 f$ is a geodesic, then $f$ is a solution of \eqref{ODE1vez}.
\end{proposition}

\begin{proof}

    Assume that $\mbox{j}^1 f$ is a geodesic curve. We can express the tangent vector field to the curve $\mbox{j}^1f$ in the frame $( e_1,e_2,e_3 )$ as
    $$
    (\mbox{j}^1 f)'(x)=e_1+(f''(x)-\phi) e_3.
    $$

    If $\mbox{j}^1 f$ is a geodesic curve, then 
    \begin{equation}\label{derivcovprol}
        \begin{aligned}
            \nabla_{(\mbox{j}^1 f)'(x)} (\mbox{j}^1 f)'(x)&=(f''(x)-\phi)^2\frac{\phi-u_1 \phi_{u_1}}{ u_1}e_1\\
            &\quad +(f''(x)-\phi)^2 \frac{\phi-u_1 \phi_{u_1}}{ u_1^2}e_2\\
            &\quad +\left(f'''(x)-e_1(\phi)-\frac{\phi}{u_1}(f''(x)-\phi)\right) e_3=0.
        \end{aligned}        
            \end{equation}
    Then, since $\big(\frac{\phi}{u_1}\big)_{u_1}\neq 0$, we have $\phi-u_1 \phi_{u_1}\neq 0$. Therefore,
            $$
            f''(x)-\phi=0,
            $$
            thus completing the proof.

\end{proof}
\begin{remark}
    Proposition \ref{geodesicas} does not apply to equations satisfying $\big(\frac{\phi}{u_1}\big)_{u_1}= 0$. For example, if we consider the equation $u_2=u_1$, it can be checked that  the function $f(x)=x+e^x$ is not a solution, but the prolongation $\mbox{j}^1 f$ is a geodesic curve, since it verifies \eqref{derivcovprol}.
    
    Nevertheless, this family of equations takes the form 
    \begin{equation}\label{odeespecial}
    u_2 = K(u) u_1,
    \end{equation}
    so they can be fully integrated by quadratures. Indeed, observe that $u_1 - \int K(u)\, du$ is a first integral of equation \eqref{odeespecial}:
    $$
e_1 \left(u_1 - \int K(u)\, du\right) = \phi - u_1 K(u) = 0.
    $$
Now, the family of first-order ODEs
    $$
    u_1 - \int K(u)\, du = C,
$$
where $C \in \mathbb{R}$, can always be solved by another quadrature, and their solutions correspond to solutions of \eqref{odeespecial}.
\end{remark}

To continue the analysis of the geometry of the manifold $(M,g)$, we compute the curvature of the connection, i.e., the matrix of 2-forms $\Omega=(\Omega^i_j)$ given by equation \eqref{curvformdef}, whose explicit expression is too involved to be included here. Nevertheless, by using equation \eqref{cartansecond}, a straightforward computation yields the following components of the Riemann curvature tensor with respect to the frame $(e_1,e_2,e_3)$:
\begin{subequations}\label{Ksec}
	\begin{align}
            R^1_{\,\,212}&=\Omega^1_{\, 2} (e_1,e_2)=\frac{1}{4}-\frac{e_1(\phi)}{u_1}, \label{Ksec1}\\
            R^1_{\,\,313}&=\Omega^1_{\, 3} (e_1,e_3)=- \frac{3}{4}+ \phi_u- \phi_{u_1}^2  - \phi \phi_{u_1u_1} - \phi_{uu_1} u_1 + \frac{3\phi \phi_{u_1}}{u_1}  - \frac{2 \phi^2}{u_1^2}, \label{Ksec2}\\
            R^2_{\,\,323}&=\Omega^2_{\, 3} (e_2,e_3)=\notag\\
            &=\frac{1}{4} - \frac{\phi \phi_{u_1} +\phi_{u u_1 }}{u_1}  - \frac{\phi \phi_{u_1 u_1}-\phi^2-\phi_u+\phi_{u_1}^2}{u_1^2}  +  \frac{4\phi \phi_{u_1}} {u_1^3}  - \frac{3\phi^2}{u_1^4}.\label{Ksec3}
    \end{align}
\end{subequations}
Recall that these components correspond to the sectional curvatures along the planes generated by the pairs $\{e_1,e_2\}$, $\{e_1,e_3\}$ and $\{e_2,e_3\}$, respectively.




\begin{example}\label{ladelrho}
    Consider the second-order ODE
    \begin{equation}\label{odeladelrho}
        u_2=\sqrt{1-\kappa u_1^2},
    \end{equation}
    \sloppy where $\kappa \neq 0$. In this case, the associated Riemannian manifold consists of ${M=\{(x,u,u_1)\in J^1(\mathbb R,\mathbb R): u_1\neq 0\}}$, together with the Riemannian metric
    $$
    g=(1 + u_1^2) dx^2 - 2 u_1 dx du + \left( 1-\kappa + \frac{1}{u_1^2} \right) du^2 - 2 \frac{\sqrt{1-\kappa u_1^2}}{u_1} du du_1 + du_1^2.
    $$
    According to the results above, we have the following sectional curvatures:
    \begin{subequations}\label{KsecEj1}
        \begin{align}
                R^1_{\,\,212}&=\frac{1}{4}+\kappa, \label{KsecEj11}\\
                R^1_{\,\,313}&=- \frac{3}{4}-\frac{2}{u_1^2}, \label{KsecEj12}\\
                R^2_{\,\,323}&=\frac{1}{4} +\frac{1}{u_1^2}-\frac{3}{u_1^4}.\label{KsecEj13}
        \end{align}
    \end{subequations}
\end{example}

\begin{example}\label{curvatura0}
The autonomous second-order ODE
\begin{equation}\label{ODEKcero}
u_2=\frac{4u_1^2+u^2+u}{2 u+1},
\end{equation}
gives rise to a Riemannian manifold which has zero sectional curvature along the planes generated by $\{e_1,e_2\}$. Indeed, from \eqref{Ksec1} we have
$$
R^1_{\,\,212}=\frac{1}{4}-\frac{e_1\left(\frac{4u_1^2+u^2+u}{2 u+1}\right)}{u_1}=0.
$$

\end{example}

The sectional curvature $R^1_{\,\,212}$ plays a significant role in understanding the geometric structure of the manifold. In particular, it is closely tied to the foliation introduced in the following section.

\section{Energy foliation}\label{SecEnergy}

We now turn our attention to a key geometric feature of the manifold $(M,g)$, namely the existence of a foliation, which we will refer to as the energy foliation, related to the integrability of the differential equation \eqref{ODE1vez}.

Consider the distribution given by
$$
D_p:=\{v\in T_pM:v\contract (\omega^3)_p=0\},
$$
for each $p\in M$. This distribution is spanned by the vector fields $\{e_1,e_2\}$. Thus, it is involutive, since
$$
[e_1,e_2]=-\frac{\phi}{u_1} e_2.
$$

By Frobenius' theorem \cite{lee2013smooth}, there exists a foliation $\mathcal{E}$ on $M$ such that the tangent space to the leaf at $p\in M$ is $D_p$. The leaves are given locally by the level sets of a certain smooth function $E:M\to \mathbb R$, i.e., in a neighborhood $V$ of $p$, the leaves are described by
$$
\Sigma_C=\{p\in V: E(p)=C\},\,\, C\in\mathbb R.
$$   
Since $T_p\Sigma_C=D_p$, we have that such a function $E$ must satisfy
\begin{equation}\label{primitiva}
    dE=\mu \omega^3,
\end{equation}
for a non-vanishing smooth function $\mu$ defined on $V$. Equivalently, $E$ must be a solution to the following homogeneous linear partial differential equation (PDE)
    \begin{equation}\label{conditionE}
        E_u+\frac{\phi}{u_1}E_{u_1} =0.
    \end{equation}

\begin{definition}\label{GenEnFol}
    The foliation $\mathcal{E}$ on $M$ will be called the energy foliation of equation \eqref{ODE1vez}, or simply the energy of the equation.
\end{definition}

\begin{remark} 
    Note that while the energy foliation $\mathcal{E}$ of an autonomous second-order ODE can be locally defined by a function $E$, it can equally well be described by any other function that is functionally dependent on $E$. As we will see in Theorem~\ref{generalEnergy}, in classical examples arising in mechanics one of these functions is, precisely, the classical notion of energy of the system, thus justifying the terminology. However, for a general autonomous second-order ODE, it is important to clarify that we cannot refer to any specific function $E$ as \emph{the energy}, since there are many such functions, and no single one is a natural or preferred choice.
\end{remark}

\begin{remark}\label{lagrang_aut}
    Given any function $E$ satisfying \eqref{conditionE} we can solve the differential equation
    \begin{equation}\label{eqLagrang}
        u_1L_{u_1}-L=E
    \end{equation}
to determine a Lagrangian $L=L(u,u_1)$ suitable for equation \eqref{ODE1vez}. This Lagrangian may be considered non-standard (see \cite{Khan2020InverseVP} and references therein). Indeed, a particular solution to \eqref{eqLagrang} is given by
\begin{equation}\label{nonSTDLag}
    L(u, u_1) = u_1 \int \frac{E(u, u_1)}{u_1^2} \, d u_1.
\end{equation}
It can be checked that the corresponding Euler-Lagrange equation for this Lagrangian is
$$
-E_u-\frac{u_2}{u_1}E_{u_1}=0,
$$
which, by virtue of \eqref{primitiva}, is equivalent to equation \eqref{ODE1vez}. This establishes the existence of local solutions to the inverse problem of the calculus of variations for autonomous second-order ODEs, a topic of ongoing relevance in mathematical physics \cite{douglas1941solution,torres2006hamiltonians,del2009hamiltonian,Khan2020InverseVP}. Importantly, the Lagrangians obtained are themselves autonomous.
\end{remark}

\begin{example}\label{EFladelrho}
In the case of equation \eqref{odeladelrho} in Example \ref{ladelrho}, we can find a function determining its corresponding energy foliation by solving the PDE \eqref{conditionE}:
\begin{equation}\label{PDEkappa}
\begin{aligned}
    u_1 E_u+\sqrt{1-\kappa u_1^2}E_{u_1}=0.
\end{aligned}
\end{equation}
It can be checked that a solution to equation \eqref{PDEkappa} is the smooth function
$$
E=u+\frac{1}{\kappa}\sqrt{1-\kappa u_1^2}.
$$
So the leaves of the energy foliation $\mathcal{E}$ of equation \eqref{odeladelrho} are the surfaces in $M$ given by
$$
\Sigma_C=\left\{\left(x,u,u_1\right) \in M : u+\frac{1}{\kappa}\sqrt{1-\kappa u_1^2}=C\right\}, \quad C\in \mathbb R.
$$
    According to Remark \ref{lagrang_aut}, and taking into account \eqref{nonSTDLag}, the smooth function
\begin{equation}\label{lagranKcte}
L(u,u_1)=-u-\frac{1}{\kappa} \left(\sqrt{1-\kappa u_1^2}+\sqrt{\kappa} u_1 \arcsin(\sqrt{\kappa}u_1)\right),
\end{equation}
in case $\kappa>0$, or the smooth function
\begin{equation}\label{lagranKcteneg}
    L(u,u_1)=-u-\frac{1}{\kappa} \left(\sqrt{1-\kappa u_1^2}+\sqrt{-\kappa} u_1 \text{arcsinh}(\sqrt{-\kappa}u_1)\right),
\end{equation}
for $\kappa<0$, is a solution to equation \eqref{eqLagrang}, so it is a non-standard autonomous Lagrangian whose Euler-Lagrange equation is \eqref{odeladelrho}.
\end{example}

\begin{example}
To find the energy foliation for the second-order ODE \eqref{ODEKcero} in Example \ref{curvatura0}, we solve the PDE
\begin{equation}\label{PDEK0}
    4(2u+1)u_1 E_u+(4 u_1^2+u^2+u)E_{u_1}=0.
\end{equation}
It can be checked that the smooth function
$$
E=\frac{ 1}{2u+1}u_1^2-\frac{4 u^2+2u+1}{32u+16}
$$
is a particular solution to \eqref{PDEK0}, and its level sets define the energy foliation of \eqref{ODEKcero}.

On the other hand, we use the expression \eqref{nonSTDLag} in Remark \ref{lagrang_aut} to obtain
$$
L=\frac{ 1}{2u+1}u_1^2+\frac{4 u^2+2u+1}{32u+16},
$$
which is an autonomous Lagrangian whose corresponding Euler--Lagrange equation is \eqref{ODEKcero}. In this case, the Lagrangian takes the form of a kinetic energy term plus a potential energy term, and the function $E$ can be considered as a mechanical energy.
\end{example}

In the rest of this section we are going to explore the geometry of the leaves of the energy foliation $\mathcal{E}$. The metric given by \eqref{metric} can be written as
$$
g=\omega^1\otimes \omega^1+\omega^2\otimes \omega^2+\omega^3\otimes \omega^3,
$$
since the coframe \eqref{coframe} is orthonormal. As $\omega^3|_{\Sigma_C} \equiv 0$, the induced metric on the surface $\Sigma_C$ is
\begin{equation}\label{g_restringida}
    \tilde{g}=\tilde{\omega}^1\otimes \tilde{\omega}^1+\tilde{\omega}^2\otimes \tilde{\omega}^2,
\end{equation}
where $\tilde{\omega}^i$ denotes the restriction of $\omega^i$ to $\Sigma_C$, for $i=1,2$.

Correspondingly, the shape operator of $\Sigma_C$ is given by
$$
S=\frac{1}{2} \tilde{\omega}^1\otimes \tilde{e}_2+\frac{1}{2} \tilde{\omega}^2\otimes \tilde{e}^1,
$$
since we have, from equation \eqref{seconfform} together with \eqref{connectionform},
$$
s_{11}=0, s_{12}=s_{21}=\frac{1}{2}, s_{22}=0.
$$
Then, we have the following result regarding the extrinsic geometry of $\Sigma_C$ within $M$:
\begin{proposition}\label{minimal}
    The leaves of the energy foliation are minimal surfaces with constant extrinsic curvature.
\end{proposition}
\begin{proof}
The mean curvature for each of these surfaces is given by 
$$
H=\frac{1}{2}(s_{11}+s_{22})=0,
$$
so they are minimal surfaces. 

On the other hand, the extrinsic Gauss curvature is 
$$
K_{\text{ext}}=s_{11}s_{22}-s_{12}s_{21}=-\frac{1}{4}.
$$    
\end{proof}

\begin{remark}
    According to this result, the energy of an autonomous second-order ODE forms a minimal foliation of the associated manifold $(M,g)$. The study of this kind of foliations is an area of significant interest in differential geometry \cite{sullivan1979homological,oshikiri1981remark,haefliger1980some,oshikiri1987some,moser2006minimal}.
\end{remark}

On the other hand, with respect to the intrinsic geometry of $\Sigma_C$, we have the following immediate consequence of equations \eqref{gausseq} and \eqref{Ksec1}:
\begin{corollary}\label{curvatureKE}
    The intrinsic Gauss curvature of $\Sigma_C$ is
\begin{equation}\label{Kintrinsic}
    K_{\text{int}}=-\frac{e_1(\phi)}{u_1}.
\end{equation}
\end{corollary}


\begin{example}
We now return to Example \ref{EFladelrho}. The Riemannian metric \eqref{g_restringida} induced in the surfaces $\Sigma_C$, obtained by means of the local parametrization of $\Sigma_C$ given by
$$
\iota(x,u)=\left(x,u,\sqrt{\frac{1}{\kappa}-\kappa(C-u)^2}\right),
$$
is
$$
\tilde{g}=\left(\frac{\kappa+1}{\kappa} - \kappa(C - u)^2\right) dx^2-2\sqrt{\frac{1}{\kappa} - \kappa(C - u)^2}dxdu+du^2.
$$
The intrinsic Gauss curvature is, according to \eqref{Kintrinsic},
\begin{equation}\label{intrinsicK}
K_{\text{int}}=-\frac{e_1(\sqrt{1-\kappa u_1^2})}{u_1}=\kappa.
\end{equation}

On the other hand, the shape operator of $\Sigma_C$ can be represented, in the coordinate frame, by the matrix
$$
S = \begin{pmatrix}
-\frac{1}{2}\sqrt{\frac{1}{\kappa}-\kappa(C-u)^2} & \frac{1}{2}-\frac{1}{2\kappa}+\frac{\kappa}{2}(C-u)^2 \\
\frac{1}{2} & \frac{1}{2}\sqrt{\frac{1}{\kappa}-\kappa(C-u)^2}
\end{pmatrix}.
$$
It can be checked that the extrinsic Gauss curvature of the leaf $\Sigma_C$ is 
$$
K_{\text{ext}}=\det(S)=-\frac{1}{4},
$$
and the mean curvature is $H=\frac{1}{2}\text{tr}(S)=0$, so it is a minimal surface, as stated in Proposition \ref{minimal}.
\end{example}

\begin{example}
    In the case of equation \eqref{ODEKcero} of Example \ref{curvatura0}, similar computations show that the leaves of its energy foliation satisfy
    $$
K_{\text{int}}=K_{\text{ext}}=-\frac{1}{4}.
    $$
\end{example}

\section{Lagrangian mechanical systems}\label{SecLagrangian}
In this section, we will focus on second-order ODEs arising as the Euler-Lagrange equations of one-dimensional mechanical systems defined by a Lagrangian function. The variable $x$ will represent time, $u$ will denote the generalized coordinate, and $u_1$ will represent the generalized velocity.

We begin by providing a justification for the terminology introduced in Definition \ref{GenEnFol}. Consider a mechanical system governed by a Lagrangian $L = L(x,u, u_1)$. For such a system, recall that the function 
\begin{equation}\label{energyfunction} 
    h(u, u_1) := u_1 L_{u_1} - L 
\end{equation} 
is known as the \emph{energy function}. Under standard assumptions for $L$, this function coincides with the total mechanical energy of the system, which is typically expressed as $h(u,u_1) = T(u,u_1) + V(u)$, where $T$ and $V$ represent the kinetic and potential energies, respectively.

In the case of an autonomous Lagrangian system, where $L = L(u, u_1)$, the energy function \eqref{energyfunction} remains conserved along the solutions of the equation of motion of the system. For further details regarding the energy function and its conservation in autonomous systems, the reader is referred to \cite[Section 2.7]{goldstein1950classical}.

Having established these preliminaries, we now state the following result:

\begin{theorem}\label{generalEnergy}
    Given an autonomous one-dimensional Lagrangian system, its energy function \eqref{energyfunction} defines the energy foliation of the corresponding equation of motion.
\end{theorem}

\begin{proof}
    The Euler-Lagrange equation for a mechanical system defined by the autonomous Lagrangian $L = L(u, u_1)$ is
$$
\frac{d}{dx} L_{u_1}  - L_u = 0
$$
    which can be expanded as
    $$
    u_1L_{uu_1}  + u_2 L_{u_1u_1}  - L_u = 0.
$$
    This is an autonomous second-order ODE, whose corresponding $\omega^3$ (see equation \eqref{coframe}) is given by 
$$
\omega^3=-\frac{L_u-u_1 L_{u u_1}}{u_1 L_{u_1 u_1}}du +du_1.
$$

In order to check that the energy function \eqref{energyfunction} satisfies condition \eqref{primitiva}, we compute the exterior derivative:
$$
dh=\left(u_1 L_{u u_1}-L_u\right) du+u_1 L_{u_1 u_1} du_1.
$$

We observe that $dh=u_1 L_{u_1 u_1}\omega^3$, so the result is proven.
\end{proof}

In the following subsections, we analyze the applicability of the results presented in this paper to specific examples of Lagrangian mechanical systems.

\subsection{Lagrangian for a particle in a gravitational field}

A classical problem in the context of Newtonian gravity involves understanding the behavior of a particle under the influence of the gravitational field created by a mass distribution. Consider a one-dimensional universe with spatial coordinate $u$, and a mass distribution defined by the smooth function $\rho(u)$. This distribution generates a gravitational field that exerts a force on a test particle of mass $m$ located at position $u$. The gravitational potential at position $u$ due to the mass distribution, denoted by $\Phi(u)$, is derived from the Poisson equation \cite{poisson2014gravity}, which in one dimension is given by:
$$
\Phi_{uu}=4 \pi G \rho(u),
$$
where $G$ is the gravitational constant.

The Lagrangian for the test particle moving in this gravitational field is then given by
$$
L = \frac{1}{2} m u_1^2 - m \Phi(u),
$$
and the corresponding equation of motion is
\begin{equation}\label{odemassdensity}
    u_2 = -\Phi_u.
\end{equation}

Within our framework, the metric for the ODE \eqref{odemassdensity} is, according to equation \eqref{metric},
\begin{equation}\label{metricmassdensity}
    g=(1 + u_1^2) dx^2 - 2 u_1 dx du + \left( 1 + \frac{\Phi_u^2}{u_1^2} \right) du^2 + 2 \frac{\Phi_u}{u_1} du du_1 + du_1^2,
\end{equation}
and the orthonormal coframe determined in \eqref{coframe} is
\begin{equation}\label{coframemassdensity}
    \begin{aligned}
        & \omega^1=dx,\\
        & \omega^2=-u_1 dx+du,\\
        & \omega^3=\frac{\Phi_u}{u_1} du+du_1.
\end{aligned}
\end{equation}

To find the energy foliation, we don't need to identify smooth functions $E$ and $\mu$ satisfying $dE=\mu \omega^3$. According to Theorem \ref{generalEnergy}, the energy foliation is described by the energy function of the system, which, in this case, coincides with the total mechanical energy:
$$
h=\frac{1}{2} m u_1^2 + m \Phi(u).
$$
The surfaces of constant energy, defined by $h(u,u_1)=C$, $C\in \mathbb R$, are minimal, according to Proposition \ref{minimal}. And their intrinsic Gauss curvature is given by equation \eqref{Kintrinsic} in Corollary \ref{curvatureKE}
\begin{equation}\label{curvaturemass}
    K_{\text{int}}=-\frac{e_1( -\Phi_u)}{u_1}=\Phi_{uu}=4 \pi G \rho(u).
\end{equation}

Thus, we see that the intrinsic Gauss curvature of the leaves of the energy foliation is directly tied to the mass density generating the gravitational field. Equation \eqref{curvaturemass} reflects the fact that the curvature is proportional to the amount of matter present in the system.

\subsection{The damped harmonic oscillator}
In this subsection, we examine the well-studied damped harmonic oscillator, a fundamental system in classical mechanics, central to understanding a wide variety of physical systems, from mechanical vibrations to electrical circuits. It describes the motion of an oscillating object subject to a restoring force proportional to its displacement, along with a damping force that opposes its velocity.

This system is governed by the second-order differential equation:
\begin{equation}\label{dampedosc}
    u_2 = -\alpha u_1 - \lambda u,
\end{equation}
where $u=u(x)$ represents the displacement as a function of time, $u_1=u_1(x)$ is the velocity, and $u_2=u_2(x)$ is the acceleration. The parameter \(\alpha\) represents the damping coefficient, which quantifies the resistance to motion (such as friction or air resistance), while \(\lambda\) is the spring constant, characterizing the strength of the restoring force. 

It is well known that this system does not present energy conservation in the classical sense. Indeed, this system is typically described by using the time-dependent Lagrangian \cite{goldstein1950classical}
$$
L = \frac{e^{\alpha x}}{2} \left(u_1^2 - \lambda u^2\right),
$$
so that the corresponding energy function \eqref{energyfunction} is
\begin{equation}\label{efdampedosc}
    h= \frac{e^{\alpha x}}{2} \left(u_1^2 + \lambda u^2\right),
\end{equation}
which is clearly not conserved over time.

However, equation \eqref{dampedosc} does not explicitly depend on time, so we can apply the ideas introduced in the previous sections. First, we consider the Riemannian metric associated to equation \eqref{dampedosc}, given by
$$
g=(1 + u_1^2) dx^2 - 2 u_1 dx du + \left( 1 +\left(\alpha+\lambda \frac{u}{u_1}\right)^2 \right) du^2 - 2 \left(\alpha+\lambda \frac{u}{u_1}\right) du du_1 + du_1^2.
$$
The orthonormal coframe introduced in \eqref{coframe} is
\begin{equation}\label{coframedamped}
    \begin{aligned}
        & \omega^1=dx,\\
        & \omega^2=-u_1 dx+du,\\
& \omega^3=\left(\alpha+\lambda\frac{u}{u_1}\right)du+du_1.
\end{aligned}
\end{equation}
To find the energy foliation $\mathcal{E}$, we write down condition \eqref{conditionE}, which in this case corresponds to the following PDE
\begin{equation}\label{pdedamped}
    u_1E_u-(\alpha u_1+\lambda u)E_{u_1}=0.
\end{equation}

In what follows, we will focus on the underdamped case, i.e., $\alpha^2<4\lambda$, but the remaining cases can be developed in an analogous manner. It can be checked that a particular solution for equation \eqref{pdedamped} in this case is given by
\begin{equation}\label{Edamped}
    E=\frac{e^{\frac{\alpha}{\omega} \arctan \left( \frac{\alpha u_1 + 2 \lambda u}{2\omega u_1} \right)}   }{2} (\alpha u u_1 +u_1^2 +  \lambda u^2),
\end{equation}
where $\omega:=\frac{1}{2}\sqrt{4\lambda-\alpha^2}$. The level sets of this function defines the energy foliation $\mathcal{E}$ of the second-order ODE \eqref{dampedosc}.

It is interesting to highlight that the function $E$ defined in \eqref{Edamped} exhibits a structural resemblance to the energy function presented in equation \eqref{efdampedosc}. Moreover, by setting $\alpha = 0$ in equation \eqref{Edamped}, which corresponds to the undamped harmonic oscillator, one recovers the classical expression for the mechanical energy of the harmonic oscillator: 
$$
E = \frac{1}{2}(u_1^2 + \lambda u^2).
$$ 
Consequently, one may regard \eqref{Edamped} as a non-mechanical form of energy that remains conserved along the solutions.

On the other hand, the leaves of the foliation $\mathcal{E}$ are, according to Proposition \ref{minimal}, minimal surfaces with $K_{\text{ext}}=-\frac{1}{4}$. Moreover, by Corollary \ref{curvatureKE}, their intrinsic curvature is
$$
K_{\text{int}}=\lambda-\alpha^2-\alpha \lambda \frac{u}{u_1}.
$$

It is worth noting that the intrinsic curvature of this system is closely related to the damping coefficient and the spring constant. In particular, for a harmonic oscillator without damping, the intrinsic curvature simplifies to the spring constant: $K_{\text{int}}=\lambda$.

Finally, observe that the energy foliation of the second-order ODE \eqref{dampedosc} can also be expressed as the level sets of the smooth function
\begin{equation}\label{Edamped2}
    \tilde{E}:=\ln \left(2 E\right)=\frac{\alpha}{\omega} \arctan \left( \frac{\alpha u_1 + 2 \lambda u}{2\omega u_1} \right)+\ln \left( \alpha u u_1 +u_1^2 +  \lambda u^2\right).
    \end{equation}
By solving the differential equation \eqref{eqLagrang} using $\tilde{E}$, we find the following non-standard autonomous Lagrangian 
\begin{equation}\label{dampedLagrangian}
    L= \frac{2 u_1}{\omega u} \arctan\left( \frac{\alpha u + 2 u_1}{2 u\omega } \right) -\frac{ \alpha }{ \omega}\arctan\left( \frac{\alpha u_1 + 2 \lambda u}{2 \omega u_1} \right) - \ln\left( \alpha u u_1 + \lambda u^2 + u_1^2 \right),
 \end{equation}
 which has previously appeared in the literature \cite{dampedHarmOsc}, although derived through a different methodology.

\section{Concluding remarks} 
In this work, we have introduced a Riemannian metric on an open subset of the first-order jet space via an autonomous second-order ODE. We studied the geometry of the resulting Riemannian manifold in relation to the integrability of the second-order ODE.
In particular, we have shown the link between geodesics and the solutions of the ODE. Furthermore, we have defined the notion of energy foliation and showed its relationship with the classical concept of energy in the case of second-order ODEs arising in mechanics. We have also analyzed the geometry of this foliation.

We believe that this approach, which has enabled us to identify autonomous Lagrangians for autonomous second-order ODEs, thus providing novel insights into the inverse problem of the calculus of variations, has the potential to lead to further developments in the field of analysis and geometry.


\bibliographystyle{ieeetr}
\bibliography{references}

\begin{thebibliography}{10}

\bibitem{cheb1999integrating}
E.~S. Cheb-Terrab and A.~D. Roche, ``Integrating factors for second-order
  odes,'' {\em Journal of Symbolic Computation}, vol.~27, no.~5, pp.~501--519,
  1999.

\bibitem{duarte2001solving}
L.~Duarte, S.~Duarte, L.~da~Mota, and J.~Skea, ``Solving second-order ordinary
  differential equations by extending the prelle-singer method,'' {\em Journal
  of Physics A: Mathematical and General}, vol.~34, no.~14, p.~3015, 2001.

\bibitem{muriel2009}
C.~Muriel and J.~L. Romero, ``First integrals, integrating factors and
  {$\lambda$}-symmetries of second-order differential equations,'' {\em J.
  Phys. A: Math. Theor.}, vol.~42, no.~36, p.~365207 (17pp), 2009.

\bibitem{yumaguzhin2010differential}
V.~A. Yumaguzhin, ``Differential invariants of second order odes, i,'' {\em
  Acta applicandae mathematicae}, vol.~109, no.~1, pp.~283--313, 2010.

\bibitem{pancinf-sym}
A.~J. Pan-Collantes, A.~Ruiz, C.~Muriel, and J.~L. Romero,
  ``$\mathcal{C}^{\infty}$-symmetries of distributions and integrability,''
  {\em J. {D}iff. {E}q.}, vol.~348, pp.~126--153, 2023.

\bibitem{duartefirstintegrals}
L.~Duarte, J.~Eiras, and L.~Da~Mota, ``An efficient way to determine
  {L}iouvillian first integrals of rational second order ordinary differential
  equations,'' {\em Computer Physics Communications}, vol.~298, p.~109088, 05
  2024.

\bibitem{al2014research}
Y.~Al-Hasan, ``Research article evaluation of matlab methods used to solve
  second order linear ode,'' {\em Research Journal of Applied Sciences,
  Engineering and Technology}, vol.~7, no.~13, pp.~2634--2638, 2014.

\bibitem{langkah2017numerical}
M.~P. U.~S. Langkah, N.~WAELEH, and Z.~A. MAJID, ``Numerical algorithm of block
  method for general second order odes using variable step size,'' {\em Sains
  Malaysiana}, vol.~46, no.~5, pp.~817--824, 2017.

\bibitem{al2020three}
M.~Al-Jawary, M.~Adwan, and G.~Radhi, ``Three iterative methods for solving
  second order nonlinear odes arising in physics,'' {\em Journal of King Saud
  University-Science}, vol.~32, no.~1, pp.~312--323, 2020.

\bibitem{santanaAutonomous}
M.~Santana, ``Exact solutions of nonlinear second-order autonomous ordinary
  differential equations: Application to mechanical systems,'' {\em Dynamics},
  vol.~3, pp.~444--467, 08 2023.

\bibitem{Bayrakdar2018a}
Z.~O. Bayrakdar and T.~Bayrakdar, ``Burgers' {E}quations in the {R}iemannian
  {G}eometry {A}ssociated with {F}irst-{O}rder {D}ifferential {E}quations,''
  {\em Advances in Mathematical Physics}, vol.~2018, pp.~1--8, 2018.

\bibitem{Bayrakdar2018}
T.~Bayrakdar and A.~A. Ergin, ``{M}inimal {S}urfaces in {T}hree-{D}imensional
  {R}iemannian {M}anifold {A}ssociated with a {S}econd-{O}rder {ODE},'' {\em
  Mediterranean Journal of Mathematics}, vol.~15, jul 2018.

\bibitem{bayrakdar2019geometric}
Z.~O. BAYRAKDAR and T.~Bayrakdar, ``A geometric description for simple and
  damped harmonic oscillators,'' {\em Turkish Journal of Mathematics}, vol.~43,
  no.~5, pp.~2540--2548, 2019.

\bibitem{bayrakdar2021curvature}
T.~Bayrakdar and Z.~O. Bayrakdar, ``The curvature property of a linear
  dynamical system,'' {\em Avrupa Bilim ve Teknoloji Dergisi}, no.~28,
  pp.~1288--1290, 2021.

\bibitem{bayrakdar2022geometry}
T.~Bayrakdar, ``Geometry of a surface in riemannian 3-manifold corresponding to
  a smooth autonomous dynamical system,'' {\em International Journal of
  Geometric Methods in Modern Physics}, vol.~19, no.~12, p.~2350024, 2022.

\bibitem{pancollantes2023surfaces}
A.~J. Pan-Collantes and J.~A. Álvarez García, ``Surfaces associated with
  first-order {ODEs}.''
  \href{https://arxiv.org/abs/2312.04489}{arXiv:2312.04489}, 2023.

\bibitem{pancollantes2024integration}
A.~J. Pan-Collantes and J.~A. Álvarez García, ``Integration of first-order
  {ODEs} by {J}acobi fields.''
  \href{https://arxiv.org/abs/2404.14352}{arXiv:2404.14352}, 2024.

\bibitem{Bayrakdar2024}
T.~Bayrakdar and A.~Turhan, ``Equivalence problem for first and second-order
  odes with a quadratic restriction,'' {\em International Journal of Geometric
  Methods in Modern Physics}, 07 2024.

\bibitem{saunders1989geometry}
D.~J. Saunders, {\em The geometry of jet bundles}, vol.~142.
\newblock Cambridge University Press, 1989.

\bibitem{olver86}
P.~J. Olver, {\em Applications of {L}ie groups to differential equations},
  vol.~107.
\newblock New York: Springer-Verlag, 1986.

\bibitem{stephani}
H.~Stephani, {\em Differential Equations: Their Solutions Using Symmetry}.
\newblock New York: Cambridge University Press, 1989.

\bibitem{chen1999lectures}
W.~Chen, S.~S. Chern, and K.~S. Lam, {\em Lectures on differential geometry},
  vol.~1.
\newblock World Scientific Publishing Company, 1999.

\bibitem{Morita}
S.~Morita, {\em Geometry of Differential Forms}.
\newblock Rhode Island: American Mathematical Society, 2001.

\bibitem{ivey2016cartan}
T.~A. Ivey and J.~M. Landsberg, {\em Cartan for Beginners}, vol.~175.
\newblock American Mathematical Soc., 2016.

\bibitem{lee2006curvature}
J.~M. Lee, {\em Riemannian manifolds: an introduction to curvature}, vol.~176.
\newblock Springer Science \& Business Media, 2006.

\bibitem{docarmoriemannian}
M.~P. Do~Carmo, {\em Riemannian geometry}, vol.~6.
\newblock Springer, 1992.

\bibitem{Carinena2023}
J.~F. Cari{\~{n}}ena and M.~C. Mu{\~{n}}oz-Lecanda, ``Geodesic and {N}ewtonian
  vector fields and symmetries of mechanical systems,'' {\em Symmetry},
  vol.~15, p.~181, jan 2023.

\bibitem{spivak1999comprehensive3}
M.~Spivak, {\em A comprehensive introduction to differential geometry. Vol. 3}.
\newblock Publish or Perish, 1999.

\bibitem{lee2013smooth}
J.~M. Lee, {\em Smooth manifolds}.
\newblock Springer, 2013.

\bibitem{Khan2020InverseVP}
B.~A. Khan, S.~Chatterjee, G.~A. Sekh, and B.~Talukdar, ``Inverse variational
  problem for nonlinear dynamical systems,'' {\em Acta Physica Polonica A},
  2020.

\bibitem{douglas1941solution}
J.~Douglas, ``Solution of the inverse problem of the calculus of variations,''
  {\em Transactions of the American Mathematical Society}, vol.~50, no.~1,
  pp.~71--128, 1941.

\bibitem{torres2006hamiltonians}
G.~Torres~del Castillo and I.~Rubalcava~Garc{\'\i}a, ``Hamiltonians and
  lagrangians of non-autonomous one-dimensional mechanical systems,'' {\em
  Revista mexicana de f{\'\i}sica}, vol.~52, no.~5, pp.~429--432, 2006.

\bibitem{del2009hamiltonian}
G.~T. del Castillo, ``The hamiltonian description of a second-order ode,'' {\em
  Journal of Physics A: Mathematical and Theoretical}, vol.~42, no.~26,
  p.~265202, 2009.

\bibitem{sullivan1979homological}
D.~Sullivan, ``A homological characterization of foliations consisting of
  minimal surfaces,'' {\em Commentarii Mathematici Helvetici}, vol.~54, no.~1,
  pp.~218--223, 1979.

\bibitem{oshikiri1981remark}
G.-i. Oshikiri, ``A remark on minimal foliations,'' {\em Tohoku Mathematical
  Journal, Second Series}, vol.~33, no.~1, pp.~133--137, 1981.

\bibitem{haefliger1980some}
A.~Haefliger, ``Some remarks on foliations with minimal leaves,'' {\em Journal
  of Differential Geometry}, vol.~15, no.~2, pp.~269--284, 1980.

\bibitem{oshikiri1987some}
G.-i. Oshikiri, ``Some remarks on minimal foliations,'' {\em Tohoku
  Mathematical Journal, Second Series}, vol.~39, no.~2, pp.~223--229, 1987.

\bibitem{moser2006minimal}
J.~Moser, ``Minimal foliations on a torus,'' in {\em Topics in Calculus of
  Variations: Lectures given at the 2nd 1987 Session of the Centro
  Internazionale Matematico Estivo (CIME) held at Montecatini Terme, Italy,
  July 20--28, 1987}, pp.~62--99, Springer, 2006.

\bibitem{goldstein1950classical}
H.~Goldstein, {\em Classical Mechanics}.
\newblock Reading, MA: Addison Wesley, 1950.

\bibitem{poisson2014gravity}
E.~Poisson and C.~M. Will, {\em Gravity: Newtonian, post-newtonian,
  relativistic}.
\newblock Cambridge University Press, 2014.

\bibitem{dampedHarmOsc}
V.~Chandrasekar, M.~Senthilvelan, and M.~Lakshmanan, ``On the lagrangian and
  hamiltonian description of the damped linear harmonic oscillator,'' {\em
  Journal of Mathematical Physics}, vol.~48, p.~032701, 03 2007.

\end{thebibliography}

\end{document}